\documentclass{amsart}

\usepackage{amsmath, amsthm, amsfonts}
\usepackage{amssymb}
\usepackage[utf8]{inputenc}
\usepackage{latexsym}
\usepackage{verbatim}
\usepackage{url}
\usepackage{textcomp}
\usepackage[all,cmtip]{xy}
\usepackage{color}
\usepackage{hyperref}
\definecolor{mylinkcolor}{rgb}{0.0,0.0,0.65}
\definecolor{mycitecolor}{rgb}{0.0,0.0,0.65}
\definecolor{myurlcolor}{rgb}{0.0,0.0,0.65}
\hypersetup{colorlinks=true,urlcolor=myurlcolor,citecolor=mycitecolor,linkcolor=mylinkcolor,linktoc=page,breaklinks=true}

\usepackage{tabularx}
\input{xypic}
\usepackage{enumerate}
\usepackage{amssymb}
\usepackage{array}
\usepackage{tikz}
\usepackage{tikz-cd}
\usepackage{rotating}
\usetikzlibrary{arrows,shapes,positioning}
\usetikzlibrary{decorations.markings}
\tikzstyle arrowstyle=[scale=1]
\usepackage{capt-of}

\DeclareMathAlphabet{\mathfr}{U}{euf}{m}{n}

\newtheorem{theorem}{Theorem}[section]
\newtheorem*{theorem*}{Theorem}

\newtheorem{proposition}[theorem]{Proposition}
\newtheorem{corollary}[theorem]{Corollary}

\newtheorem{lemma}[theorem]{Lemma}
\newtheorem{question}[theorem]{Question}

\theoremstyle{remark}

\newtheorem{definition}[theorem]{Definition}
\newtheorem{example}[theorem]{Example}

\definecolor{myblue}{rgb}{.2,.6,.75}
\definecolor{mygreen}{rgb}{.4,.7,.4}

\newcommand{\Q}{\mathbb Q}

\newcommand{\Gal}{\mathrm{Gal}}

\newcommand{\Z}{\mathbb Z}

\newcommand{\F}{\mathbb F}

\newcommand{\End}{\operatorname{End}}
\newcommand{\Hom}{\operatorname{Hom}}

\newcommand{\Aut}{\operatorname{Aut}}

\newcommand{\Br}{\mathrm{Br}}

\newcommand{\Res}{\operatorname{Res}}

\newcommand{\Tra}{\operatorname{Tra}}

\newcommand{\Image}{\mathrm{Im}}
\newcommand{\Kernel}{\mathrm{Ker}}

\newcommand{\norm}[1]{{\mathcal{N}_{#1}}}

\numberwithin{equation}{section}

\usepackage{url}

\newcommand*{\stgroup}[2][]{\href{https://www.lmfdb.org/SatoTateGroup/#2}{{\ifx&#1& #2 \else #1 \fi}}}

\begin{document}
\title[]{Geometrically simple counterexamples to a local-global principle for quadratic twists}

\author{Emiliano Ambrosi}

\address{Institut de Recherche Math\'ematique Avanc\'ee (IRMA), Universit\'e de Strasbourg, 7 Rue Ren\'e Descartes, 67000 Strasbourg}
\email{eambrosi@unistra.fr}
\urladdr{http://emiliano.ambrosi.perso.math.cnrs.fr/}
\author{Nirvana Coppola}

\address{Dipartimento di Matematica ``Tullio Levi-Civita'', Università di Padova, Via Trieste 63, 35131 Padova}
\email{ncoppola@math.unipd.it}
\urladdr{https://sites.google.com/view/nirvanacoppola/home}
\author{Francesc Fit\'e}

\address{Departament de matem\`atiques i inform\`atica and Centre de recerca matem\`atica,
Universitat de Barcelona,
Gran via de les Corts Catalanes 585, 08007 Barcelona}
\email{ffite@ub.edu}
\urladdr{http://www.ub.edu/nt/ffite/}

\begin{abstract}
Two abelian varieties $A$ and $B$ over a  number field $K$ are said to be strongly locally quadratic twists if they are quadratic twists at every completion of $K$. While it was known that this does not imply that $A$ and $B$ are quadratic twists over $K$, the only known counterexamples (necessarily of dimension $\geq 4$) are not geometrically simple. We show that, for every prime $p\equiv 13 \pmod{24}$, there exists a pair of geometrically simple abelian varieties of dimension $p-1$ over $\Q$ that are strongly locally quadratic twists but not quadratic twists. The proof is based on Galois cohomology computations and class field theory.
\end{abstract}
\maketitle
\tableofcontents
\section{Introduction}
\subsection{Twists and local twists}
Let $K$ be a number field, write $\Gamma_K$ for its absolute Galois group, denote by $\Sigma_K$ the set of finite places of $K$, and for $v\in \Sigma_K$ write $K_v$ for the corresponding completion and $K(v)$ for the residue field. If $n\in \mathbb N$, we denote with $\zeta_n$  a primitive $n^{th}$-root of unity.

Let $A$ and $B$ be abelian varieties defined over a number field $K$. A celebrated theorem of Faltings \cite{Fal83} shows that if the reductions of $A$ and $B$ are isogenous over $K(v)$ for a density one set of $v\in \Sigma_K$, then $A$ and $B$ are isogenous over $K$.

Various variants of this result have been then studied (see e.g. \cite{CT22, Fit24, FP24, KL20, Raj98, Ram00}). In particular, one can show that if the reductions of $A$ and $B$ are isogenous over $\overline{K(v)}$ for a density one set of $v\in \Sigma_K$, then $A$ and $B$ are isogenous over $\overline K$ (see for example \cite{KL20,CT22}).

We will work in the category of abelian varieties up to isogeny. In particular, we will say that $A$ and $B$ are \emph{twists} if there exists a finite Galois extension $F$ of $K$ such that the base changes $A_F$ and $B_F$ are isogenous. The result in the previous paragraph naturally raises the question of whether the nature of a twist of $A$ is determined by that of a density one set of its reductions. Different incarnations of this problem have been studied (e.g. \cite{Fit24,FP24}). In this paper we continue this study, focusing on the situation of quadratic twists.

\subsection{Quadratic twists and locally quadratic twists}

Let us write $\Aut(A_{\overline K})$ to denote $(\End(A_{\overline K})\otimes \Q)^\times$. Recall that the set of twists of $A$ is in a canonical bijection with the Galois cohomology group $H^1(\Gamma_K, \Aut(A_{\overline K}))$, where $\Gamma_K$ is the absolute Galois group of $K$. For $\alpha\in  H^1(\Gamma_K, \Aut(A_{\overline K}))$, write $A_{\alpha}$ for the corresponding twist of $A$. We say that $A_{\alpha}$ is a \emph{quadratic twist} of $A$ if $\alpha$ is in the image of $H^1(\Gamma_K, \{\pm 1 \})\rightarrow H^1(\Gamma_K, \Aut(A_{\overline K}))$, i.e. if $\alpha$ is the image of a continuous character $\chi : \Gamma_K\rightarrow \{\pm 1\}$.

We say that $A$ and $B$ are \emph{locally quadratic twists} if their reductions are quadratic twists over $K(v)$ for a density one set of $v\in \Sigma_K$. So, in the spirit of the results of \cite{KL20,CT22}, one would like an answer to the following question:
\begin{question}\label{qu : mainquestion}
	If $A$ and $B$ are locally quadratic twists, are they quadratic twists?
\end{question}
If $\End(A_{\overline K})=\Z$, the above question admits a positive answer (see \cite{Fit24}). While it is known that the answer is negative in general (see Section \ref{sec: GWintro} for more details), only non geometrically simple counterexamples (of dimensions 4 and 6) were known prior to the present work (see \cite[Rem. 4.10, \S6.2]{Fit24}). The main result of this paper is a strong negative answer to Question \ref{qu : mainquestion} for geometrically simple abelian varieties of arbitrarily big dimension.
\begin{theorem}\label{thm : mainlocal}
Fix a prime $p\equiv 13 \pmod {24}$. Every geometrically simple abelian variety of dimension $p-1$ over $\Q$ such that $A_{\overline \Q}$ has complex multiplication by $\Q(\zeta_{3p})$ has a twist which is locally quadratic but not quadratic.
\end{theorem}
Such abelian varieties exist by \cite[Theorem 3.0.1]{GGL24}.
\subsection{The Grunwald-Wang counterexample}\label{sec: GWintro}
Question \ref{qu : mainquestion} has already been studied by the third named author in \cite{Fit24}. There, it is proven that it has a positive answer if $\dim(A)\leq 3$, and a counterexample is given in dimension 4. To motivate the strategy for the proof of Theorem \ref{thm : mainlocal}, let us recall this counterexample. It consists of the pair of abelian fourfolds $A$ and $B$, which are the Jacobians of the genus 4 curves over $\Q$ given by the affine models
$$
C:y^ 2 = x^9+x, \qquad C':y^2=x^ 9 +16 x .
$$
The curves $C$ and $C'$ were found via a computer search. The proof given in \cite{Fit24} that $A$ and $B$ are locally quadratic twists involved the explicit computation of the Weil polynomials of $A$ and $B$ via Jacobi sums. The proof that they are not quadratic twists combined the fact that the minimal extension over which all homomorphisms between $A$ and $B$ are defined is $\Q(\zeta_{16}, \sqrt[8]{16})$ with the fact that $A$ and $B$ are not quadratic twists over any of the three quadratic subfields of this extension. This required the computation of Frobenius traces at prescribed primes.

As pointed out to us by Alex Smith, the Grunwald--Wang theorem \cite[Chap. X]{AT68} suggests a more conceptual proof: on the one hand, using that 16 admits an $8^{th}$-root $\alpha_p$ modulo every odd prime $p$, one can build the isomorphism $\phi_{\alpha_p}: (x,y)\mapsto (\alpha_p x , \alpha^{9/2}_py)$ over $\overline{\mathbb F}_p$ between the reductions of $C$ and $C'$, showing that $A$ and $B$ are quadratic twists modulo every odd prime; on the other hand, exploiting the fact that 16 does not admit an $8^{th}$-root in $\Q$, one can show that $A$ and $B$ are not quadratic twists as explained in Section \ref{sec : GW}.

As mentioned before, we remark that $A$ and $B$ are not geometrically simple. This can be shown by observing that they have potential complex multiplication by $\mathbb Q(\zeta_{16})$ but non-primitive CM type.
\subsection{Strongly locally quadratic twists}
We observe that, by Hensel's lemma, $16$ has an $8^{th}$-root $\beta_p$ not only over $\F_p$, but also over $\Q_p$, for every odd prime $p$. Hence $A$ and $B$ have the stronger property of being quadratic twists over $\Q_p$ for all odd $p$.\footnote{One can also show that such $A$ and $B$ are quadratic twists over the completion at every place of $\Q(\sqrt{7})$, see Example \ref{ex : all places}} This leads to the following definition.
\begin{definition}
	Let $A$ and $B$ be two abelian varieties over $K$. We say that $A$ and $B$ are strongly locally quadratic twists if they are quadratic twists over $K_v$ for a density one set of places $v$ of $K$.
\end{definition}
Clearly, if $A$ and $B$ are strongly locally quadratic twists, then they are locally quadratic twists, but we do not know if the converse holds. With this definition, we can state the following stronger version of Theorem \ref{thm : mainlocal}.
\begin{theorem}\label{thm : mainstronglylocal}
Fix a prime $p\equiv 13 \pmod{24}$. Every geometrically simple abelian variety over $\Q$ of dimension $p-1$ such that $A_{\overline \Q}$ has complex multiplication by $\Q(\zeta_{3p})$ has a twist which is strongly locally quadratic but not quadratic.
\end{theorem}
We also show (Proposition \ref{prop : minimality}) that $n=39$ is the minimal odd $n$ for which there exists a pair of abelian varieties over $\Q$ with potential complex multiplication by $\Q(\zeta_{n})$ which are strongly locally quadratic twists but not quadratic twists. Similar techniques can be used to show that a geometrically simple abelian variety over $\Q$ with potential complex multiplication by $\Q(\zeta_{20})$ could be twisted in order to obtain a counterexample in dimension 4, but we do not know if such abelian varieties exist. Note that our source \cite[Theorem 3.0.1]{GGL24} of geometrically simple abelian varieties over $\Q$ with potential complex complex multiplication by $\Q(\zeta_{n})$ requires $n$ odd or $n\equiv 2 \pmod 4$.
\subsection{Final remark}
We would like to end this introduction by stressing the importance of shifting from the notion of locally quadratic twist to that of strongly locally quadratic twist. While group representation techniques are well suited for the study of locally quadratic twists, the cohomological approach faces the difficulties of the composition of maps
$$
H^1(\Gamma_K,\Aut(A_{\overline K}))\rightarrow  H^1(\Gamma_{K_v},\Aut(A_{\overline K}))\rightarrow H^1(\Gamma_{K(v)},\Aut(A_{\overline{K(v)}})).
$$

Since the natural map $\Aut(A_{\overline K})\rightarrow \Aut(A_{\overline K_v})$ is an isomorphism (see \cite[Corollary 12.13]{EvdGM}), the study of strongly locally quadratic twists reduces to the study of the first and more accessible of the above composition of maps. 

In contrast, while the specialisation map $\Aut(A_{\overline K_v}) \rightarrow \Aut(A_{\overline{K(v)}})$ is injective (as follows from the injection in \cite[Theorem 12.10]{EvdGM}), it is rarely surjective. For example if $A$ is a non CM elliptic curve over $\Q$ or $\Q_p$, then $\Aut(A_{\overline \Q})= \Aut(A_{\overline \Q_p}) \simeq \Q^\times$, but $\Aut(A_{\overline{\F}_p})$ are the invertibles in either a quadratic field or a quaternion algebra, see \cite[Chapter V.3]{Sil86}). 

The adequacy of the cohomological tools for the study and explicit computation of the first map is what ultimately allowed the construction of the counterexamples presented in this article.

\textbf{Acknowledgements.} Fité thanks Université de Strasbourg for its warm hospitality during visits in May 2022, May 2023, and June 2024. Fité was financially supported by the Ramón y Cajal fellowship RYC-2019-027378-I, by the María de Maeztu Program CEX2020-001084-M, and by the AEI grant PID2022-137605NB-I00. Coppola was supported by the ANR-CYCLADES project at Université de Strasbourg. Coppola is a member of the INdAM group GNSAGA. Thanks to Jordan Ellenberg, Alex Smith, and Marco Streng for fruitful conversations. Thanks to Giuseppe Ancona for triggering the collaboration between the authors.
\section{Cohomological characterisation of strongly locally quadratic twists}
The main result of this section is Proposition  \ref{prop : counterexample_conditions}, that translates the study and the construction of strongly locally quadratic twists to a purely cohomological statement.
\subsection{Statements}\label{sec : statements}
Let $K$ be a number field and let $A/K$ be a geometrically simple abelian variety whose geometric endomorphism algebra $\End(A_{\overline K})\otimes \Q$ is a number field $E$.
In this section we give a cohomological characterisation (Proposition \ref{prop :  counterexample_conditions}) of the existence of twists of $A$ which are strongly locally quadratic but not (globally) quadratic.
Write $K\subseteq L$ for the minimal extension over which all the endomorphisms of $A$ are defined. It is a finite and Galois extension.
 Let $G$ be the Galois group of $L/K$ and consider the following commutative diagram with exact rows and columns.
	\begin{center}
	\begin{tikzcd}
	&& H^1(\Gamma_K, \{ \pm 1\}) \arrow{d}& \\
		&& H^1(\Gamma_K, E^{\times}) \arrow{d}& \\
		1 \arrow{r} & H^1(G, E^{\times}/\{\pm 1\}) \arrow{r} & H^1(\Gamma_K, E^{\times}/\{\pm 1 \}) \arrow{r} \arrow{d}{\delta} &H^1(\Gamma_L, E^{\times}/\{\pm 1\})\\
		&& H^2(\Gamma_K, \{ \pm 1\}) &
	\end{tikzcd}

\end{center}
 The vertical sequence is induced by the exact sequence of $G$-modules
 $$1\rightarrow \{\pm 1\}\rightarrow E^{\times}\rightarrow E^{\times}/\{\pm 1\}\rightarrow 1$$
 and the horizontal one by inflation and restriction.
 For an element $x\in H^1(G, E^{\times}/\{\pm 1\})$
 consider the following conditions:
 \begin{enumerate}[(i)]
 	\item $x\neq 1$;
 	\item $x$ restricts to $1$ in $H^1(C, E^{\times}/\{\pm 1\})$ for every cyclic subgroup $C\subseteq G$;
 	\item $x$ maps to $1$ in $H^2(\Gamma_K, \{\pm 1\})$.
 \end{enumerate}

\begin{proposition}\label{prop : counterexample_conditions}
	The following are equivalent:
\begin{enumerate}
	\item there exists a twist $B$ of $A$ which is strongly locally quadratic but not quadratic;
	\item there exists an element $x \in H^1(G, E^{\times}/\{\pm 1\})$ satisfying $(i)-(ii)-(iii)$ above.
\end{enumerate}
\end{proposition}
\begin{corollary}\label{cor : cylic}
	If $G$ is cyclic then every strongly locally quadratic twist is quadratic.
\end{corollary}
We will also need a straightforward variant of Proposition \ref{prop : counterexample_conditions}.
We say that a subgroup $C\subseteq G$ is \emph{maximally cyclic} if it is cyclic and every subgroup $H\subseteq G$  that properly contains $C$ is not cyclic. Since condition (ii) is clearly equivalent to:
\begin{enumerate}
\item[(ii')]  $x$ restricts to $1$ in $H^1(C, E^{\times}/\{\pm 1\})$ for every maximally cyclic subgroup $C\subseteq G$,
\end{enumerate}
we can restate Proposition \ref{prop : counterexample_conditions} as follows.
\begin{proposition}\label{prop : conditionvariant}
	The following are equivalent:
\begin{enumerate}
		\item there exists a twist $B$ of $A$ which is strongly locally quadratic but not quadratic;
			\item there exists an element $x \in H^1(G, E^{\times}/\{\pm 1\})$ satisfying $(i)-(ii')-(iii)$ above.
\end{enumerate}
\end{proposition}

Before giving the proof of Proposition \ref{prop :  counterexample_conditions}, we state a general lemma that will be useful in the rest of the paper. 

 \begin{lemma}\label{lem : tokendstrong}
If $K=L$, then every strongly locally quadratic twist of $A$ is a quadratic twist.
\end{lemma}
\begin{proof}

Since all endomorphisms of $A$ are defined over $K$, we also have that all the endomorphisms of $A$ are defined over $K_v$ for every $v\in \Sigma_K$. Hence

$$H^1(\Gamma_K, E^{\times})=\Hom(\Gamma_K,E^{\times})\quad \text{and}\quad H^1(\Gamma_{K_v}, E^{\times})=\Hom(\Gamma_{K_v},E^{\times}),$$ and the maps
$$H^1(\Gamma_K, \{\pm 1\})\rightarrow H^1(\Gamma_K, E^{\times})\quad \text{and}\quad H^1(\Gamma_{K_v}, \{\pm 1\})\rightarrow H^1(\Gamma_{K_v}, E^{\times})$$
are injective.
Let $B$ be a strongly locally quadratic twist of $A$ corresponding to an element $\chi\in H^1(\Gamma_K, E^{\times})$. It is enough to show that $\Image(\chi)=\{\pm 1\}$. But this holds on every decomposition group by assumption, hence it holds on all of $\Gamma_K$, since decomposition groups form a dense subset of $\Gamma_K$.
\end{proof}
\subsection{Proof}
We now prove Proposition \ref{prop : counterexample_conditions}.
We start by proving that (1) implies (2).
 Let $\tilde x\in H^1(\Gamma_K,E^{\times})$ be the cohomology class associated to $B$ and let $x$ be its image in $H^1(\Gamma_K,E^{\times}/\{\pm 1\})$. By construction, $\delta(x)=1$ and, since $B$ is not a quadratic twist of $A$, we have that $x\neq 1$. Since $A$ and $B$ are strongly locally quadratic twists, by Lemma \ref{lem : tokendstrong}, they are quadratic twists over $L$, hence the restriction  of $x$ in $H^1(\Gamma_L,E^{\times}/\{\pm 1\})$ is trivial, and thus $x\in H^1(G,E^{\times}/\{\pm 1\})$.
We are left to show that $x$ is trivial when restricted to any cyclic subgroup $C\subseteq G$. By Chebotarev, for every cyclic subgroup $C\subseteq G$ there is a positive density set of finite places $v$ of $K$ (unramified in $L$) such that the decomposition group $D_v$ is $C$. In particular we can choose one $v$ such that $A_v$ and $B_v$ are quadratic twists over $K_v$, so the restriction of $x$ to $H^1(\Gamma_{K_v}, E^{\times}/\{\pm 1\})$ is trivial. The conclusion follows from the commutative diagram
\begin{equation}\label{diagram: inf res to dec}
	\begin{tikzcd}
		 H^1(G, E^{\times}/\{\pm 1\})  \arrow[hook]{r} \arrow{d} & H^1(\Gamma_K, E^{\times}/\{\pm 1\}) \arrow{d} \\
		 H^1(D_v, E^{\times}/\{\pm 1\})  \arrow[hook]{r} & H^1(\Gamma_{K_v}, E^{\times}/\{\pm 1\}),
	\end{tikzcd}
\end{equation}
since the bottom horizontal arrow is injective.

We now prove that (2) implies (1). Let $x$ be as in the statement. Since the top horizontal map of \eqref{diagram: inf res to dec} is injective, the image of $x$ in $H^1(\Gamma_K, E^{\times}/\{\pm 1 \})$, that we still denote by $x$, is nontrivial. By assumption (iii), one has $\delta(x)=1$ (where $\delta$ is the connecting homomorphism defined in Section \ref{sec : statements}), so $x$ lifts to an element $\tilde x$ in $H^1(\Gamma_K, E^{\times})$. Thus $\tilde x$ defines a twist $A_{\tilde x}$ of $A$, which is not quadratic since $x \ne 1$.
Let $\Sigma$ be the set of finite places of $K$ which are not ramified in $L$, so that, for every $v\in \Sigma$, the decomposition group $D_v$ is cyclic. Since $\Sigma$ consists of all but finitely many places of $K$, it is enough to show that for every  $v\in \Sigma$, the twist $A_{\tilde x,v}$ of $A_v$ is quadratic. For this, it is enough to show that the restriction of $ x\in H^1(\Gamma_K,E^{\times}/\{\pm 1\})$ to  $ H^1(\Gamma_{K_v},E^{\times}/\{\pm 1\})$ is trivial. Since $D_v$ is cyclic, this follows from assumption (ii) and the commutative diagram \eqref{diagram: inf res to dec}.

	\section{Geometrically simple counterexamples}
In this section, after some group cohomology preliminaries, we prove Theorem~\ref{thm : mainstronglylocal}. We then discuss the minimality of our counterexample for $n=39$ among abelian varieties with potential complex multiplication by $\Q(\zeta_n)$ for odd $n$.
\subsection{Preliminaries}
\subsubsection{Cohomology of cyclic groups}
Let $C$ be a finite cyclic group of cardinality $n$ acting on an abelian group $M$, written multiplicatively, and write $g\in C$ for a generator. Let $M^C\subseteq M$ be the group of elements that are fixed by $C$ and $\norm{C}:M\rightarrow M$ be the norm map, sending  $m\in M$ to $mg(m)\dots g^{n-1}(m)$. Recall from \cite[VIII, $\S$4]{Ser89} that one has natural identifications:

\begin{equation}\label{equation: cohom identifications}
H^i(C; M)= \begin{cases}M^C \qquad &\text{ if } i=0, \\
\Kernel(\norm{C})/\langle g(m)m^{-1}\rangle_{m\in M} \quad  &\text{ if } i \text{ is  odd,} \\
M^{C}/\Image(\norm{C}) \quad  &\text{ if } i \text{ is  even.}
\end{cases}
\end{equation}

Under these identifications, if
$$1\rightarrow N\rightarrow M\rightarrow Q\rightarrow 1$$
is an exact sequence of $C$-abelian groups, then
the connecting morphism
$$\delta:\Kernel(\norm{C})/\langle g(q)q^{-1}\rangle_{q\in Q}\simeq H^1(C,Q)\rightarrow H^2(C,N)\simeq N^{C}/\Image(\norm{C})$$
has the following description: for $x\in \Kernel(\norm{C})$, choose a lift $\widetilde x\in M$ and set $\delta(x):=\norm{C}(\widetilde x)$.

Finally, for $i>0$ we note that $H^i(C,\{\pm 1\})\simeq \{\pm 1\}$ if $C$ is of even order and $H^i(C,\{\pm 1\})=1$ otherwise.\subsubsection{Preliminary lemmas}
Assume that $E/F$ is a finite Galois extension of number fields with Galois group $G$ and consider the short exact sequence of $G$-modules
 \begin{equation}\label{equation: ses F mod pm1}
 1\rightarrow \{\pm 1\}\rightarrow E^\times\rightarrow E^\times/\{\pm 1\}\rightarrow 1 .
 \end{equation}
 \begin{lemma}\label{lem : injectivitycyclic}
 \begin{enumerate}
 \item[]
 	\item The connecting morphism $H^1(G, E^\times/\{\pm 1\})\rightarrow H^2(G,\{\pm 1\})$ is injective. If $G$ is cyclic, it is induced by the norm map under the identifications \eqref{equation: cohom identifications}.
 	\item There is a natural short exact sequence
$$1 \rightarrow F^\times/\{\pm 1 \}\rightarrow (E^\times/\{\pm 1\})^G\rightarrow H^1(G,\{\pm 1\})\rightarrow 1.$$
 \end{enumerate}

 \end{lemma}
\proof
\begin{enumerate}
\item[]
	\item Injectivity follows from Hilbert's 90th theorem, while the description of the map in the case of a cyclic group follows from the previous discussion.

	\item This follows again from Hilbert's 90th theorem applied to the exact sequence in cohomology induced by \eqref{equation: ses F mod pm1}.
 \endproof
\end{enumerate}
Assume now that $K$ is a number field and that $A$ is a geometrically simple abelian variety over $K$ whose geometric endomorphism algebra $\End(A_{\overline K})\otimes \Q$ is a number field $E$. We let $K\subseteq L$ be the minimal extension over which all the endomorphisms of $A$ are defined. By Galois theory, $\Gal(L/K)$ is isomorphic to $\Gal(E/F)$, where $F=\End(A)\otimes \Q$.
\begin{lemma}\label{lem : reductionodddegree}
	Let $K\subseteq F\subseteq L$ be an intermediate Galois extension  corresponding to a normal subgroup $H\subseteq G=\Gal(L/K)$. Let $x\in H^1(G,E^{\times})$ and suppose that $[F:K]$ is odd. Then $x$ satisfies the conditions (i),(ii),(iii) of Proposition \ref{prop : counterexample_conditions} if and only its restriction to $H$ does.
	 \end{lemma}
	\proof
	Consider the commutative diagram
\begin{center}
	\begin{tikzcd}[column sep = small]
H^1(G, E^{\times}/\{\pm 1\})\arrow[hook]{r}\arrow{d}& H^2(\Gamma_K, \{\pm 1\})\arrow{d}\\
H^1(H, E^{\times}/\{\pm 1\})\arrow[hook]{r}& H^2(\Gamma_F, \{\pm 1\}).
\end{tikzcd}
	\end{center}

Since $[G : H]$ is odd, by \cite[Prop. 1.6.9]{NSW00} the rightmost vertical map is injective, so conditions (i) and (iii) are equivalent for $H$ and $G$.

To prove that also condition (ii) is equivalent for $H$ and $G$, observe first that condition (ii) for $G$ clearly implies the one for $H$.
Conversely, assume now that condition (ii) holds for $H$, and let $C\subseteq G$ be a cyclic subgroup. Consider the commutative diagram
	\begin{center}
	\begin{tikzcd}[column sep = small]
H^1(G, E^{\times}/\{\pm 1\})\arrow{r}\arrow{d}& H^1(C, E^{\times}/\{\pm 1\})\arrow{d}\arrow[hook]{r} &H^2(C,\{\pm 1\})\arrow{d}\\
H^1(H, E^{\times}/\{\pm 1\})\arrow{r}& H^1(C\cap H, E^{\times}/\{\pm 1\})\arrow[hook]{r} & H^2(C\cap H,\{\pm 1\})
\end{tikzcd}
	\end{center}
	where the injectivity of the rightmost horizontal arrows is due to Lemma \ref{lem : injectivitycyclic}.
	Since $C/C\cap H\simeq CH/H\subseteq G/H$ is of odd order, the rightmost vertical map is also injective, thus the vanishing of the restriction of $x$ to $H^1(C\cap H, E^{\times}/\{\pm 1\})$ implies the vanishing of the restriction of $x$ to $H^1(C, E^{\times}/\{\pm 1\})$.
		\endproof
\subsection{An infinite family of counterexamples}
We start recalling the statement of our main result.
\begin{theorem}\label{theorem : highdimension}
Let $p$ be an odd prime such that $p\equiv 13 \pmod{24}$.
Let $A$ be a geometrically simple abelian variety over $\Q$ such that $A_{\overline \Q}$ has complex multiplication by $E:=\Q(\zeta_{3p})$.
Then there exists a twist of $A$ which is strongly locally quadratic but not quadratic.
\end{theorem}
Since $A$ has potential complex multiplication by $E$ and it is geometrically simple, by \cite[Prop. 26 and Prop. 28, Chap. II]{Shi98}, the reflex field of $E$ is $E$ itself. Then, by \cite[Prop. 20.4]{Shi98} or \cite[Prop. 5.17]{Shi71} or \cite[Thm. 1.1, Chap. 3]{Lan83}, the field $E$ is also the minimal field of definition of all the endomorphisms of $A$.

We write $G:=\Gal(E/\Q)$ and $G_1:=\Gal(E/\Q(\sqrt{-3}))$, so that $G/G_1\simeq \Z/2 \Z$. Let $\sigma$ be the projection of complex conjugation to $G/G_1$. It naturally acts on $(E^\times /\{\pm 1\})^{G_1}$.

Since $p \equiv 1 \pmod 3$, it splits in $\Q(\sqrt{-3})$. Hence there exist $a,b\in \mathbb \Q$ such that $a^2+3b^2=3p$.

We let $y:=\dfrac{a+b\sqrt{-3}}{\sqrt{-3p}}\in E^{\times}/\{\pm 1\}$ and we observe that $y\in (E^\times/\{\pm 1\})^{G_1}$. By construction, one has $\sigma(y)y=-1$, so we can associate to $y$ a cohomology class $x\in H^1(G/G_1, (E^\times/\{\pm 1\})^{G_1})$. By inflation, we get a cohomology class $x\in H^1(G,E^\times/\{\pm 1\})$. By Proposition \ref{prop : counterexample_conditions}, it is then enough to show that $x$ satisfies conditions (i), (ii') and (iii) from Section \ref{sec : statements}.

\subsubsection{Condition (i)}
Since the inflation map is injective, it is enough to show that $x$ is nontrivial in $H^1(G/G_1, (E^\times/\{\pm 1\})^{G_1})$.
By Lemma \ref{lem : injectivitycyclic}, there is a short exact sequence of $G/G_1$-modules
$$1\rightarrow \Q(\sqrt{-3})^{\times}/\{\pm 1\}\rightarrow (E^{\times}/\{\pm 1\})^{G_1}\rightarrow H^1(G_1,\{\pm 1\})\rightarrow 1.$$
Since $G_1$ is cyclic of even order, we have that $H^1(G_1,\{\pm 1\})\simeq \{\pm 1\}$. Since $y\in (E^{\times}/\{\pm 1\})^{G_1}$ is not in the image of $\Q(\sqrt{-3})^{\times}/\{\pm 1\}\rightarrow (E^{\times}/\{\pm 1\})^{G_1}$, it maps non-trivially in $H^1(G_1,\{\pm 1\})$. Since $G/G_1$ is cyclic of even order and it acts trivially on $H^1(G_1,\{\pm 1\})\simeq \{\pm 1\}$, one has
$$
H^1(G/G_1,H^1(G_1,\{\pm 1\}))\simeq H^1(G/G_1,\{\pm 1\})\simeq \{\pm 1\}\simeq H^1(G_1,\{\pm 1\}).
$$
The composition of the map
$$
H^1(G/G_1,(E^{\times}/\{\pm 1\})^{G_1})\rightarrow H^1(G/G_1,H^1(G_1,\{\pm 1\}))
$$
with the above isomorphism sends $x$ to the image of $y$ under the map $(E^{\times}/\{\pm 1\})^{G_1}\rightarrow H^1(G_1,\{\pm 1\})$. We have seen that this image is nontrivial, and hence $x$ is nontrivial as well.
\subsubsection{Condition (ii')} \label{proof : cond2}
Let $\Q\subseteq k$ be the maximal subextension of $\Q \subseteq E$ of odd degree. By Lemma \ref{lem : reductionodddegree}, proving (ii') is equivalent to proving it after restricting to $k$. We will denote by $R$ the maximal totally real subfield $k(\zeta_{3p}+\overline{\zeta_{3p}})$ of $E$.
Since $p\not \equiv 1 \pmod 8$ and $p \equiv 1 \pmod 4$, the group  $H:=\Gal(E/k)$ is isomorphic to $\mathbb Z/2\Z\times \mathbb Z/4\Z$, and we have the following diagram of intermediate extensions.
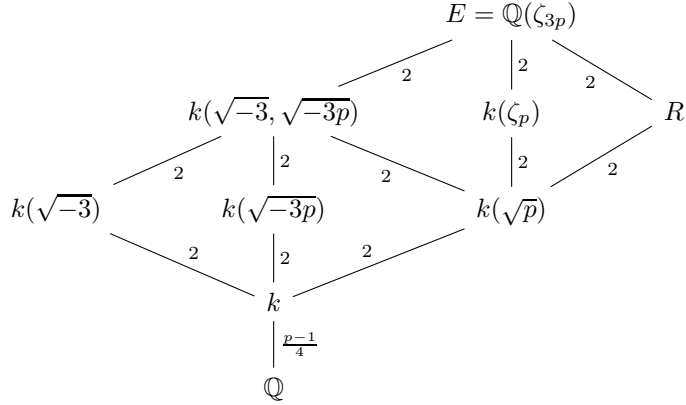
\begin{center}
	\begin{tikzcd}
		&& E=\mathbb{Q}(\zeta_{3p}) \arrow[dash]{dl}{2} \arrow[dash]{d}{2} \arrow[dash]{dr}[swap]{2} & \\
		& k(\sqrt{-3},\sqrt{-3p}) \arrow[dash]{dl}{2} \arrow[dash]{d}{2} \arrow[dash]{dr}[swap]{2} & k(\zeta_{p}) \arrow[dash]{d}{2} & R \arrow[dash]{dl}{2}\\
		k(\sqrt{-3}) \arrow[dash]{dr}{2} & k(\sqrt{-3p}) \arrow[dash]{d}{2} & k(\sqrt{p}) \arrow[dash]{dl}[swap]{2} & \\
		& k \arrow[dash]{d}{\frac{p-1}{4}}& &&\\
		& \mathbb{Q} & &
	\end{tikzcd}
\end{center}

There are only four maximally cyclic subgroups $H_1,\dots , H_4 \subseteq H$, corresponding to the four extensions of $k$ depicted in the diagram below.
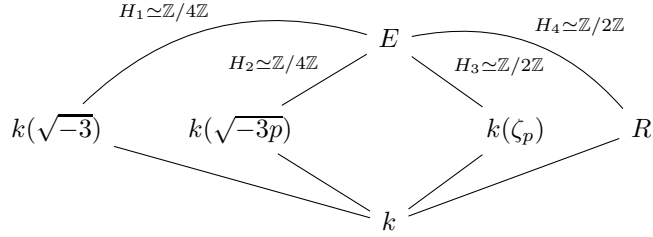
\begin{center}
	\begin{tikzcd}
		&& E  & \\
		k(\sqrt{-3}) \arrow[bend left, dash]{rru}{H_1\simeq \Z/4\Z} & k(\sqrt{-3p}) \arrow[dash]{ru}{H_2\simeq\Z/4\Z} & &k(\zeta_{p})\arrow[swap, dash]{lu}{H_3\simeq\Z/2\Z} & R \arrow[bend right, dash, swap]{llu}{H_4\simeq\Z/2\Z}\\
		&& k\arrow[dash]{ull}\arrow[dash]{ul}\arrow[dash]{rru} \arrow[dash]{ru}& &&\\

	\end{tikzcd}
\end{center}
We write $\sigma_i$ for a generator of $H_i$.
Since, by construction, $x$ is inflated from $H/H_1$, it goes to $1$ in  $H^1(H_1, E^{\times}/\{\pm 1\})$.

Recall from Lemma \ref{lem : injectivitycyclic} that the norm map
$$\norm{H_i}: H^1(H_{i},E^{\times}/\{\pm 1\})\rightarrow H^2(H_{i},\{\pm 1\})\simeq \{\pm 1\}$$
is injective. Hence it will suffice to show that
$\norm{H_i}(y)=1$ for $i=2,3,4$.

Since $H_2$ fixes $k(\sqrt{-3p})$, one has $\sigma_2(y)=\dfrac{a-b\sqrt{-3}}{\sqrt{-3p}}$. In particular, we have
$$\norm{H_2}(y)=\prod_{i=0}^3 \sigma_2^i(y)=(y \sigma(y))^2=\left(\dfrac{3p}{-3p}\right)^2=1.$$

For $i=3,4$, we see that $\sigma_i$ acts nontrivially on both $\sqrt{-3}$ and $\sqrt{-3p}$, thus sending $y$ to its complex conjugate. Therefore $\norm{H_i}(y)=y\sigma_i(y)=1$.
\subsubsection{Condition (iii)}
Recall that, for every field $M$, we have $H^2(\Gamma_M,\{\pm 1\})\simeq \Br(M)[2]$. Hence, the fundamental exact sequence of class field theory yields a short exact sequence

$$1\rightarrow H^2(\Gamma_{\mathbb{Q}}, \{\pm 1\}) \rightarrow H^2(\Gamma_{\mathbb{R}},\{\pm 1\})\times \prod_{q\in \Sigma_\Q} H^2(\Gamma_{\mathbb{Q}_q},\{\pm 1\})\xrightarrow{\sum \text{res}_q} \frac{1}{2}\Z/\Z\rightarrow 0,$$
where $\text{res}_\infty: H^2(\Gamma_\mathbb{R},\{\pm 1\}) \rightarrow \frac{1}{2}\Z/\Z$ and $\text{res}_q: H^2(\Gamma_{\mathbb{Q}_q},\{\pm 1\})\rightarrow \frac{1}{2}\Z/\Z$ are the residue morphisms.
Hence, it is enough to show that the class of $x$ becomes trivial in $H^2(\Gamma_\mathbb{R},\{\pm 1\})$ and in $H^2(\Gamma_{\mathbb{Q}_q},\{\pm 1\})$ for all primes $q$, except at most one.

To check the infinite place, write $C:=\Gal(E/R)$ for the Galois group of the maximal totally real subfield $R\subseteq E$. Since $C$ is cyclic, the restriction of $x$ to $H^1(C,E^{\times}/\{\pm 1\})$ is trivial by Section \ref{proof : cond2}.
Then the commutative diagram
\begin{center}
	\begin{tikzcd}
		H^1(G,E^{\times}/\{\pm 1\})\arrow{r}\arrow{d} & 	H^1(\Gamma_{\mathbb{Q}},E^{\times}/\{\pm 1\})\arrow{r} \arrow{d} &H^2(\Gamma_\Q,\{\pm 1\})\arrow{d}\\
		H^1(C,E^{\times}/\{\pm 1\})\arrow{r} & 	H^1(\Gamma_{\mathbb R},E^{\times}/\{\pm 1\})\arrow{r} &H^2(\Gamma_{\mathbb R},\{\pm 1\})\\
			\end{tikzcd}
\end{center}
shows that the image of $x$ in $H^2(\Gamma_{\mathbb{R}},\{\pm 1\})$ is trivial.

We claim that, for all primes $q \neq 3$, the decomposition group $D_q\subseteq G$ is cyclic. If $q\neq p$, this follows from the fact that $q$ is unramified in $E$. To check that also $D_p$ is cyclic,  we first observe that $$\left(\dfrac{-3}{p}\middle)=\middle(\dfrac{3}{p}\middle)=\middle(\dfrac{p}{3}\right)=1,$$
by quadratic reciprocity and the assumption  $p\equiv 1 \pmod 4$. Hence $p$ splits in $\Q(\sqrt{-3})$, so that $D_p\simeq \Gal(\Q(\zeta_p)/\Q)\simeq \Z/(p-1)\Z$ is cyclic.

In particular, the restriction of $x$ to $H^1(D_q,E^{\times}/\{\pm 1\})$ vanishes by Section \ref{proof : cond2}.
Hence the commutative diagram
\begin{center}
	\begin{tikzcd}
		H^1(G,E^{\times}/\{\pm 1\})\arrow{r}\arrow{d} & 	H^1(\Gamma_{\mathbb{Q}},E^{\times}/\{\pm 1\})\arrow{r} \arrow{d} &H^2(\Gamma_\Q,\{\pm 1\})\arrow{d}&\\
		H^1(D_q,E^{\times}/\{\pm 1\})\arrow{r} & 	H^1(\Gamma_{\mathbb{Q}_q},E^{\times}/\{\pm 1\})\arrow{r} &H^2(\Gamma_{\mathbb{Q}_q},\{\pm 1\})&\\
			\end{tikzcd}
\end{center}
shows that the image of $x$ in $H^2(\Gamma_{\mathbb{Q}_q},\{\pm 1\})$ is trivial.  \endproof


\subsection{Minimality of the counterexample}
In the previous section, we constructed an infinite family of geometrically simple counterexamples to Question \ref{qu : mainquestion}. In particular, when $p=13$, we obtain one example in dimension $12$. We now show that such a counterexample is the one of smaller dimension among all geometrically simple abelian varieties with geometric complex multiplication by $\Q(\zeta_n)$ for odd $n$. To do so, let $n$ be an odd number such that $\phi(n)<24$ and let $A$ be an abelian variety as above.

If $(\Z/n\Z)^{\times}$ is cyclic (i.e. $n$ is a power of an odd prime), then no twist of $A$ can yield a counterexample by Corollary \ref{cor : cylic}. This leaves only three possibilities for $n$, namely $15=3\cdot 5$, $21=3\cdot 7$ and $33=3\cdot 11$.
All of these are excluded by the following proposition.
\begin{proposition}\label{prop : minimality}
Let $p$ be a prime.	Let $A$ be an abelian variety over $\Q$ such that $A_{\overline \Q}$ has complex multiplication by $E:=\mathbb Q(\zeta_{3p})$ and is geometrically simple.
Assume that either
\begin{enumerate}
 \item $p\equiv 2 \pmod 3$, or
	\item $p\equiv 3 \pmod 4$.
\end{enumerate}
	Then, every strongly locally quadratic twist of $A$ is a (global) quadratic twist of $A$.
\end{proposition}
\proof
Since $A$ is geometrically simple and has potential complex multiplication by~$E$, by \cite[Prop. 5.17]{Shi71} or \cite[Thm. 1.1, Chap. 3]{Lan83}, the field $E$ is also the minimal field of definition of all the endomorphisms of $A$. Let $G:=\Gal(E/\Q)$.
By Proposition \ref{prop : counterexample_conditions}, it is enough to show that there are no non trivial elements in $H^1(G,E^{\times}/\{\pm 1\})$ which vanish on every cyclic subgroup. To show this, it suffices to prove that there exists at least one cyclic subgroup $H\subseteq G$ such that $H^1(G/H,(E^{\times}/\{\pm 1\})^H)=1$, since then restriction to $H$ is injective by inflation-restriction.
For this we will use the following lemma.

\begin{lemma}\label{lemma: H1zero}
Let $H\subseteq G$ be a cyclic index two subgroup corresponding to an imaginary quadratic field $ F=\Q(\sqrt{-d})$ for $d\in \Q^\times$. Then
\begin{enumerate}
	\item  $H^1(G/H, F^\times/\{\pm 1\}) =1$.
    \item  Suppose there exists $x\in (E^\times/\{\pm 1\})^H$ such that $\norm{G/H}(x)\not \in \norm{G/H}(F^\times)$, then $H^1(G/H, (E^{\times}/\{\pm 1\})^H) =1$.\end{enumerate}
\end{lemma}

\begin{proof}

We first consider part (1). By Hilbert's 90th theorem, the sequence $1 \rightarrow \{\pm 1 \} \rightarrow F^\times \rightarrow F^\times/\{\pm 1\} \rightarrow 1$ induces a short exact sequence
	\begin{center}
		\begin{tikzcd}[column sep = tiny]
		 1\rightarrow H^1(G/H, F^\times/\{\pm 1\}) \arrow{r} & H^2(G/H,\{\pm 1\}) \arrow{r} & H^2(G/H,F^\times).
		\end{tikzcd}
	\end{center}
Hence, using the identification \eqref{equation: ses F mod pm1}, it is enough to show that the natural map
	$$\{\pm 1\}\simeq H^2(G/H,\{\pm 1\}) \rightarrow H^2(G/H,F^\times)\simeq (F^\times)^{G/H}/\norm{G/H} (F^\times)$$
	is injective. This follows from the fact that  $-1$ is not a norm for the imaginary field extension $F/\mathbb Q$.

To prove part (2), let us write $M$ to denote the quotient of $(E^\times/\{\pm 1\})^H$ by $F^\times/\{\pm 1\}$. The exact sequence from Lemma \ref{lem : injectivitycyclic}
	\begin{center}
		\begin{tikzcd}
			1 \arrow{r} & F^\times/ \{\pm 1\} \arrow{r} & (E^{\times}/\{\pm 1\})^{H} \arrow{r} & H^1(H,\{\pm 1\}) \arrow{r} & 1,
		\end{tikzcd}
	\end{center}
	shows that $M$ is isomorphic to $H^1(H,\{\pm 1\})\simeq \{\pm 1\}$, so that that, in particular, $\Kernel(\norm{G/H}:M\rightarrow M)=M$. The above exact sequence,
	together with part (1) of the lemma, induces an exact sequence
	$$		1 \rightarrow H^1(G/H,(E^{\times}/\{\pm 1\})^{H}) \rightarrow H^1(G/H, M) \stackrel{\delta}{\rightarrow} H^2(G/H,F^\times/\{\pm 1\}),
	$$
	and thus it suffices to show that $\delta$ is injective. On the one hand, observe that $H^1(G/H, M)\simeq \{\pm 1\}$. On the other hand, by \eqref{equation: ses F mod pm1}, if $\sigma$ is the nontrivial element of $G/H$, we have
	$$
	H^1(G/H, M)\simeq \Kernel(\norm{G/H})/ \langle \sigma(m)m^{-1}\rangle_{m\in M}=M/ \langle \sigma(m)m^{-1}\rangle_{m\in M}.
	$$
	Using the identifications from \eqref{equation: ses F mod pm1}, we may rewrite $\delta$ as
	$$\norm{G/H}:M/ \langle \sigma(m)m^{-1}\rangle_{m\in M}\simeq \{\pm 1\} \rightarrow  (F^\times)^{G/H}/\norm{G/H}(F^\times).
	$$
	From this, we see that $\delta(-1)=\norm{G/H}(x)$ and the injectivity of $\delta$ follows from the hypothesis that $\norm{G/H}(x)\not \in \norm{G/H}(F^\times)$.

\end{proof}
We now come back to the proof of Proposition \ref{prop : minimality}. Suppose first that $p\equiv 2 \pmod 3$. Recall that $G\simeq \Z/2\Z\times \Z/(p-1)\Z$. Let $G_1 \simeq \Z/(p-1)\Z$ be the cyclic subgroup of $G$ fixing $F_1:=\Q(\sqrt{-3})$. We claim that $H^1(G/G_1,(E^{\times}/\{\pm 1\})^{G_1})=1$.
 By \cite[Exercise 2.1]{Was97} one has that either $\sqrt{-p}\in E^{\times}$ or $\sqrt{-3p}\in E^{\times}$. Hence, by Lemma \ref{lemma: H1zero}, it is enough to show that either $p=\norm{G/G_1}(\sqrt{-p})$ or $3p=\norm{G/G_1}(\sqrt{-3p})$ is not a norm for the field extension $F_1/\Q$. Since $3$ is a norm, this amounts to showing that $p$ is not a norm. The latter follows from the condition $p\equiv 2 \pmod 3$, which is equivalent to the primality of the ideal $p\mathcal O_{F_1}$.

Suppose now that $p\equiv 3 \pmod 4$.
 In this case $\sqrt{-p}\in E^{\times}$ by \cite[Exercise 2.1]{Was97}, and the extension
 $F_2:=\Q(\sqrt{-p})\subseteq E$ is cyclic. Let $G_2$ be its cyclic Galois group.
 We claim that $H^1(G/G_i,(E^{\times}/\{\pm 1\})^{G_i})$ is trivial for either $i=1$ or $i=2$.
By Lemma \ref{lemma: H1zero}, it is enough to show that either $3=\norm{G/G_2}(\sqrt{-3})$ is not a norm in $F_2/\Q$  or that $p=\norm{G/G_1}(\sqrt{-p})$ is not a norm in $F_1/\Q$. This amounts to showing that either
$$
\left(\dfrac{-3}{p}\right)=-1 \qquad \text{or}\qquad \left(\dfrac{-p}{3}\right)=-1.
$$
This is implied by the fact that the product of the above two Legendre symbols is -1, as a consequence of quadratic reciprocity and the assumption $p\equiv 3 \pmod 4$.

\endproof

\section{Grunwald-Wang style counterexamples}
In this section we generalise the Grunwald-Wang counterexample from \cite{Fit24} recalled in Section \ref{sec: GWintro}.
Let $K$ be a field and $A$ be an abelian variety over $K$. Let $K \subseteq L$ be the minimal extension over which all the endomorphisms of $A$ are defined and let $G:=\Gal(L/K)$. Write $E$ for $\End(A_{L})\otimes \Q$.
\subsection{Preliminaries}
Let $B$ be an abelian variety over $K$ such that
$B_{L}$ is the quadratic twist of $A_{L}$ by a quadratic character $\chi$ of $\Gamma_{L}$.
Let $L\subseteq L_{\chi}$ be the quadratic extension cut by $\chi$. Since $K\subseteq L_\chi$ is the field over which all the endomorphisms of $A\times B$ are defined, it is a Galois extension of $K$ whose Galois group is a central extension
\begin{equation}\label{eq : centralextension}
	1\rightarrow \{\pm 1\}\rightarrow \Gal(L_{\chi}/K)\rightarrow  G\rightarrow 1.
\end{equation}
\begin{proposition}\label{proposition: transgressionstrong}
If $A$ and $B$ are quadratic twists then (\ref{eq : centralextension}) splits, so that
$$\Gal(L_{\chi}/K)\simeq G\times \{\pm 1\}.$$
\end{proposition}
\begin{proof}
Recall that $H^2(G, \{\pm 1\} )$ is in bijection with the set of isomorphisms classes of central extensions of the form
$1\rightarrow \{\pm 1\}\rightarrow H\rightarrow G\rightarrow 1$.

Let $\Tra: H^1(\Gamma_{L},\{\pm 1\} )\rightarrow H^2(G, \{\pm 1\} )$
be the transgression map induced by the inclusion of the normal subgroup $\Gamma_{L}\subseteq \Gamma_K$. Under the aforementioned bijection, the class of $\Tra(\chi)$ identifies with the Galois group of the Galois extension $L_{\chi}/K$.
Hence, it is enough to show that if $B$ is a quadratic twist of $A$, then $\Tra(\chi)=1$.

By \cite[Proposition 1.6.7]{NSW00}, we have a commutative diagram
\begin{center}
\begin{tikzcd}
H^1(\Gamma_K, \{\pm 1\} ) \arrow{r}{\Res}\arrow{d}{\iota} & H^1(\Gamma_{L},\{\pm 1\})^G \arrow{d}{\iota'}\arrow{r}{\Tra}& H^2(G, \{\pm 1\})\\
 H^1(\Gamma_K, E^ \times )\arrow{r}{\Res}&H^1(\Gamma_{L},E^{\times} )^G
	\end{tikzcd}
\end{center}
whose first row is exact. Observe that $\iota'$ is injective, since the action of $\Gamma_{L}$ on $E^ \times$ is trivial and then  $H^1(\Gamma_K, E^ \times )$ identifies with the quotient of $\Hom(\Gamma_K, E^ \times)$ modulo conjugation.

Let $c_{B}$ be the element in $H^1(\Gamma_K,E^{\times} )$ corresponding to $B$. Note that $\Res(c_{B})=\iota'(\chi)$.
If $A$ and $B$ are quadratic twists, then there exists a quadratic character $\tilde \chi$ of $\Gamma_K$ such that $\iota(\tilde \chi)=c_{B}$. But then
$$
\iota' (\Res(\tilde \chi))=\Res(\iota(\tilde \chi))=\Res(c_{B})=\iota'(\chi).
$$

The injectivity of $\iota'$ implies that $\chi=\Res(\tilde\chi)$ and hence $\Tra(\chi)=\Tra(\Res(\tilde \chi))=1$.
\end{proof}

\subsection{Grunwald-Wang counterexamples}\label{sec : GW}
Let $m$ be a positive integer. Assume that $E$ contains $\mathbb Q(\zeta_{2m})$. In the following, we let $\mu_{n}$ denote the group of all $n^{th}$ roots of unity. For $\alpha\in K^{\times}$, let $[\alpha]\in H^1(\Gamma_K, \mu_{2m})\simeq K^{\times}/(K^{\times})^{2m}$ the corresponding class. Define $A_{\alpha}$ as the abelian variety corresponding to the image of $[\alpha]$ through the map
$H^1(\Gamma_K,\mu_{2m})\rightarrow H^1(\Gamma_K,E^{\times})$.

\begin{example}\label{example: Kummer}
Let $A$ be the Jacobian of the curve given by the affine model $C:y^2=x^{m+1}+x$. The action of $\mu_{2m}$ on $C$ given by $(x,y)\mapsto (\zeta_m x, \zeta_{2m}y)$ yields an inclusion of $\mathbb Q(\zeta_{2m})$ in $E$. If $C_\alpha$ denotes the curve given by $y^2=x^{m+1}+\alpha x$, then
$$
\phi_\alpha: C\rightarrow C_\alpha\,,\qquad \phi_\alpha(x,y)=(\alpha^{1/m}x, \alpha^{(m+1)/(2m)}y)
$$
is an isomorphism over $\overline K$. The map $\xi_\alpha:\Gamma_K\rightarrow \mu_{2m}$ defined as $\xi_\alpha(s):=\phi_\alpha^{-1}\circ {}^s\phi_\alpha$ is a 1-cocycle. An easy calculation shows that the class of $\xi_\alpha$ corresponds to $[\alpha]$ under the Kummer isomorphism. Hence $A_\alpha$ is the Jacobian of the curve $C_\alpha$.
\end{example}

\begin{proposition}
Assume that $L=K(\zeta_{2m})$, that the only roots of unity contained in $K$ are $\pm 1$, and that there exists a finite subset $S\subseteq \Sigma_K$ such that $\alpha \in K_v^{\times,m}$ for all $v\not \in S$.
	\begin{enumerate}
		\item The abelian varieties $A$ and $A_{\alpha}$ are strongly locally quadratic twists. More precisely, $A$ and $A_\alpha$ are quadratic twists over $K_v$ for all $v\not \in S$.
		\item If $A$ and $A_{\alpha}$ are quadratic twists then either $\alpha$ or $-\alpha$ is an $m^{th}$ power.
	\end{enumerate}
\end{proposition}
\begin{proof}
Consider the exact sequence
$$1\rightarrow \mu_2\rightarrow \mu_{2m}\xrightarrow{(-)^2} \mu_m\rightarrow 1$$
where the first arrow is the natural inclusion. For any field extension $K\subseteq F$, one has an induced short exact sequence in cohomology
$$H^1(\Gamma_F, \mu_2)\rightarrow H^1(\Gamma_F, \mu_{2m})\rightarrow H^1(\Gamma_F, \mu_m),$$
which, in turn, identifies with the exact sequence
$$F^{\times}/(F^{\times})^2\xrightarrow{(-)^m} F^{\times}/(F^{\times})^{2m}\rightarrow F^{\times}/(F^{\times})^{m},$$
where the last arrow is the natural projection.
This shows that $[\alpha]_F\in H^1(\Gamma_F, \mu_{2m})$ is in the image of $H^1(\Gamma_F, \mu_2)$ as soon as $\alpha$ is an $m^{th}$-power in $F$, since this implies that it becomes trivial in $H^1(\Gamma_F, \mu_m)$. Taking $F$ to be $K_v$ for any $v\not \in S$, and considering the commutative diagram
\begin{center}
\begin{tikzcd}
H^1(\Gamma_F, \mu_2)\arrow{rr}\arrow{rd} && H^1(\Gamma_F, \mu_{2m})\arrow{ld}\\
&H^1(\Gamma_F, E^{\times}),
	\end{tikzcd}
\end{center}
 we obtain $(1)$.

 We are left to prove that if $A$ and $A_{\alpha}$ are quadratic twists then either $\alpha$ or $-\alpha$ is an $m^{th}$ power.
By the Grunwald-Wang theorem \cite[p. 96]{AT68}, since $\zeta_{2m}\in L$ and $\alpha$ is locally an $m^{th}$ power, it becomes an $m^{th}$ power in $L$, so that $L(\sqrt[2m]\alpha)/L$ is of degree two and $A_L$ and $A_{\alpha,L}$ are quadratic twists by the character defining the extension $L\subseteq L(\sqrt[2m]\alpha)$.

By Proposition \ref{proposition: transgressionstrong}, we see that
$$
\Gal(K(\zeta_{2m},\sqrt[2m]\alpha)/K)=\Gal(L(\sqrt[2m]\alpha)/K)\simeq \Gal(L/K)\times \{\pm 1\}
$$
is abelian. Since the only roots of unity contained in $K$ are $\pm 1$, by \cite[Theorem 2]{Sch77} we deduce that $\alpha^2$ is a $2m^{th}$ power, hence that one of $\alpha$ or $-\alpha$ is $m^{th}$ power. This concludes the proof.
\end{proof}
\begin{example}\label{ex : all places}
	In Example \ref{example: Kummer} , take $K=\mathbb Q(\sqrt 7)$, $m=8$, and $\alpha=16$. One has that $L=K(\zeta_{16})$. Then $A$ and $A_\alpha$ are everywhere strongly locally quadratic, but not globally quadratic, since neither $16$ or $-16$ are $8^{th}$-powers in $K$ but $16$ is an $8^{th}$-power in every localization of $K$ (see \cite[p. 98]{AT68}). It would be interesting to determine for which other values of $m$ one has $L=K(\zeta_{2m})$.
\end{example}

\end{document}